\documentclass[a4paper,reqno]{amsart}
	\usepackage{amssymb, a4wide}
\usepackage[all]{xy}
\usepackage{amsmath,amsthm,amscd}
\usepackage{tikz-cd}
\usepackage{enumitem}
\usepackage{mathtools}
\usepackage{graphicx}
\usepackage{bbold}
\usepackage{multicol}
\usepackage{xcolor}
\newtheorem{theorem}{Theorem}[section]
\newtheorem{proposition}[theorem]{Proposition}

\theoremstyle{definition}
\newtheorem{definition}[theorem]{Definition}
\newtheorem{example}[theorem]{Example}

\theoremstyle{remark}
\newtheorem{remark}[theorem]{Remark}

\numberwithin{equation}{section}

\newcommand{\vp}{\varphi}
\newcommand{\lam}{\lambda}
\newcommand{\ka}{\kappa}
\newcommand{\om}{\omega}
\newcommand{\ol}{\overline}
\newcommand{\kk}{\ensuremath{\mathbb{K}}}

\DeclareMathOperator{\Ann}{Ann}

\DeclareMathOperator{\Tetra}{Tetra}
\DeclareMathOperator{\st}{st}

\DeclareMathOperator{\Z}{Z}

\DeclareMathOperator{\LB}{\mathtt{Lb}}
\DeclareMathOperator{\XLB}{\mathtt{XLb}}
\DeclareMathOperator{\AS}{\mathtt{As}}
\DeclareMathOperator{\XAS}{\mathtt{XAs}}

\DeclareMathOperator{\U}{\mathtt{U}}
\DeclareMathOperator{\XU}{\mathtt{XU}}
\DeclareMathOperator{\Ud}{\mathtt{U_d}}
\DeclareMathOperator{\XUd}{\mathtt{XU_d}}
\DeclareMathOperator{\I}{\mathtt{I}}
\DeclareMathOperator{\J}{\mathtt{J}}

\DeclareMathOperator{\Ph}{\mathtt{\Phi}}
\DeclareMathOperator{\Ps}{\mathtt{\Psi}}
\DeclareMathOperator{\Xliea}{\mathtt{XLie_{As}}}
\DeclareMathOperator{\Xliel}{\mathtt{XLie_{Lb}}}
\DeclareMathOperator{\Liea}{\mathtt{Lie_1}}
\DeclareMathOperator{\Liel}{\mathtt{Lie_2}}

\def\Ker{\operatorname{Ker}}

\def\Hom{\operatorname{Hom}}

\def\As{\operatorname{\textbf{\textsf{Alg}}}}
\def\XAs{\operatorname{\textbf{\textsf{XAlg}}}}
\def\XLie{\operatorname{\textbf{\textsf{XLie}}}}
\def\Lie{\operatorname{\textbf{\textsf{Lie}}}}
\def\Lb{\operatorname{\textbf{\textsf{Lb}}}}
\def\XLb{\operatorname{\textbf{\textsf{XLb}}}}

\def\Di{\operatorname{\textbf{\textsf{Dias}}}}
\def\XDi{\operatorname{\textbf{\textsf{XDias}}}}

\def\C{\operatorname{\textbf{\textsf{C}}}}

\def\Der{\operatorname{Der}}
\def\D{\operatorname{D}}
\def\id{\operatorname{id}}

\DeclareMathOperator{\Set}{\textbf{\textsf{Set}}}
\DeclareMathOperator{\SplExt}{SplExt}
\DeclareMathOperator{\Aut}{Aut}

\def\Act{\operatorname{Act}}

\begin{document}

\title[Actor of a crossed module of dialgebras via tetramultipliers]{Actor of a crossed module of dialgebras via tetramultipliers}

\author[J. M. Casas]{Jos\'e Manuel Casas}
\address[Jos\'e Manuel Casas]{Departmento de Matemática Aplicada I, Universidade de Vigo, E.\ E.\ Forestal, E--36005 Pontevedra, Spain}
\email{jmcasas@uvigo.es}
\author[R. Fern\'andez-Casado]{Rafael Fern\'andez-Casado}
\address[Rafael Fern\'andez-Casado]{Departmento de Álgebra, Universidade de Santiago de Compostela\\
E--15782 Santiago de Compostela, Spain}
\email{rapha.fdez@gmail.com}
\author[X. Garc\'ia-Mart\'inez]{Xabier Garc\'ia-Mart\'inez}
\address[Xabier García-Martínez]{Departamento de Matemáticas, Esc.\ Sup.\ de Enx.\ Informática, Campus de Ourense, Universidade de Vigo, E--32004 Ourense, Spain \& \newline
	Faculty of Engineering, Vrije Universiteit Brussel, Pleinlaan 2, B--1050 Brussel, Belgium}
\email{xabier.garcia.martinez@uvigo.gal}
\author[E. Khmaladze]{Emzar Khmaladze}
\address[Emzar Khmaladze]{The University of Georgia, Kosrava St. 77$^a$, 0171 Tbilsi, Georgia 
 \&
	\newline
A. Razmadze Mathematical Institute, Tbilisi State University, Tamarashvili St.\ 6, 0177 Tbilisi, Georgia}
\email{e.khmal@gmail.com}

\thanks{This work was supported by Ministerio de Economía y Competitividad (Spain), with grant number MTM2016-79661-P. The third author is a Postdoctoral Fellow of the Research Foundation–Flanders (FWO).}

\begin{abstract}
	We study the representability of actions in the category of crossed modules of dialgebras via tetramultipliers. We deduce a pair of dialgebras in order to construct an object which, under certain circumstances, is the actor (also known as the split extension classifier). Moreover, we give give a full description of actions in terms of equations. Finally, we check that under the aforementioned circumstances, the center coincides with the kernel of the canonical map from a crossed module to its actor.
\end{abstract}
\maketitle

\section{Introduction}
Crossed modules play a very relevant role in different areas of mathematics, particularly in homotopy theory (see, for instance~\cite{Br}). Ever since the first appearance of the concept in the late 1940s due to Whitehead work~\cite{Wh}, in the particular context of groups, their presence and importance have lately extended to many different scenarios, not just as a tool but as an algebraic object with self relevance. Just as examples of the huge variety of works involving crossed modules in different structures, the reader may check~\cite{BaCr} and~\cite{KassLo} for the case of Lie algebras, Shammu’s PhD thesis~\cite{Sha} for associative algebras or~\cite{LoPi} and~\cite{ShZh} for Leibniz algebras, among many other contributions on the topic.

Bearing in mind~\cite{BaCr} and~\cite{ShZh}, it is worth to mention that crossed modules are a particular kind of strict categorification. More precisely, the equivalence has been proved in~\cite{BrSpen} for the case of groups, in~\cite{BaCr} for Lie algebras, in~\cite{Khma} for associative algebras, and in~\cite{FCPhD} for Leibinz algebras and dialgebras.

It is well-known that internal categories of groups are equivalent to crossed modules. This is true for all categories of groups with operations~\cite{Por} and in fact, this equivalence has been generally understood in all semi-abelian categories~\cite{Jan}, although it is not true in weaker algebraic contexts~{\cite{MaMoSo, Pat}}.

In the semi-abelian framework, actions are just classes of split extensions. In groups, Lie algebras and their crossed modules, the functor $\Act(-, X) \colon \C \to \Set$ that sends an object~$B$ to the set of actions of $X$ over $B$ is representable, i.e., it is naturally isomorphic to $\Hom(-, [X])$ for some object $[X]$; being $\Aut(X)$ in the case of groups and $\Der(X)$ in the case of Lie algebras. In general, the object $[X]$ is known as the actor, and a category where all its objects have an actor is called action representable~\cite{BoJaKe}. This is, in fact, a very strong property: it characterizes the variety of Lie algebras over all non-abelian varieties of non-associative algebras over an infinite field~\cite{GTVV} (another characterization was also given by the right adjointness of the kernel functor $\SplExt_{B}(\C) \to \C$, see~\cite{GaVa2, GaVa}). In other algebraic structures, it became a very interesting question to find which kind of objects play the role of actor, at least under some conditions.

In~\cite{CaInKhLa}, it is given an extension to crossed modules of the adjunction between the unit group functor and the group algebra functor. Additionally, the $2$-dimensional generalizations of the corresponding adjunctions for $\Lie$ vs $\As$ and $\Lb$ vs $\Di$ are presented in {\cite{CFCKhL, FCPhD}}, where the resulting commutative squares of categories and functors are assembled into four parallelepipeds containing the original adjunctions and their natural generalizations:

\[
\begin{tikzcd}[row sep=30, column sep=30]
	& \As \ar[shift right=2]{rrr}{\perp}[swap]{\Liea} \ar[dashed, shift left=3]{dd}[near start]{\I'_i}[swap, near start]{\dashv} \ar[shift left=1]{dl}[rotate=-140, swap]{\subset}[rotate=130]{\vdash} & & & \Lie \ar[shift right=3]{lll}[swap]{\U} \ar[shift right=1]{dd}[near start]{\vdash}[near start, swap]{\I_i} \ar[shift right=1]{dl}[rotate=-50]{\vdash}[rotate=-140]{\subset}\\
	\Di \ar[shift left=1.5]{rrr}{\LB}[swap]{\top} \ar[shift left=2]{dd}[near start]{\J'_i}[swap, near start]{\dashv} \ar[shift left=4]{ur}[rotate=40, near end]{\AS \ } & & & \Lb \ar[shift left=2.5]{lll}{\Ud} \ar[shift right=2]{dd}[near start, swap]{\J_i}[near start]{\vdash} \ar[shift right=3.5]{ur}[rotate=40, near start, swap]{\Liel} &\\
	& \XAs \ar[dashed,shift left=2.5]{uu}[near end]{\Ph'_i} \ar[dashed, shift right=1]{rrr}{\perp}[swap]{\Xliea} \ar[dashed, shift left=2]{dl}[rotate=-140, swap]{\subset}[rotate=130]{\vdash} & & & \XLie \ar[shift right=3]{uu}[near end, swap]{\Ph_i}\ar[dashed, shift right=3]{lll}[swap]{\XU} \ar[shift right=1]{dl}[rotate=-50]{\vdash}[rotate=-140]{\subset}\\
	\XDi \ar[shift left= 3]{uu}[near end]{\Ps'_i} \ar[shift left=2]{rrr}{\XLB}[swap]{\top} \ar[dashed, shift left=2.5]{ur}[rotate=40, near end]{\XAS \ } & & & \XLb \ar[shift right=3]{uu}[near end, swap]{\Ps_i}\ar[shift left=2]{lll}{\XUd} \ar[shift right=3]{ur}[rotate=40, near start, swap]{\Xliel} &
\end{tikzcd}
\]
\[
\begin{tikzcd}[row sep=30, column sep=30]
	& \As \ar[shift right=2]{rrr}{\perp}[swap]{\Liea} \ar[dashed, shift right=3]{dd}[near start, swap]{\I'_i}[near start]{\dashv} \ar[shift left=1]{dl}[rotate=-140, swap]{\subset}[rotate=130]{\vdash} & & & \Lie \ar[shift right=3]{lll}[swap]{\U} \ar[shift left=3]{dd}[near start, swap]{\vdash}[near start]{\I_i} \ar[shift right=1]{dl}[rotate=-50]{\vdash}[rotate=-140]{\subset}\\
	\Di \ar[shift left=1.5]{rrr}{\LB}[swap]{\top} \ar[shift right=2]{dd}[near start, swap]{\J'_i}[near start]{\dashv} \ar[shift left=4]{ur}[rotate=40, near end]{\AS \ } & & & \Lb \ar[shift left=2.5]{lll}{\Ud} \ar[shift left=2]{dd}[near start]{\J_i}[near start, swap]{\vdash} \ar[shift right=3.5]{ur}[rotate=40, near start, swap]{\Liel} &\\
	& \XAs \ar[dashed,shift right=2.5]{uu}[near end, swap]{\Ph'_{i+1}} \ar[dashed, shift right=1]{rrr}{\perp}[swap]{\Xliea} \ar[dashed, shift left=2]{dl}[rotate=-140, swap]{\subset}[rotate=130]{\vdash} & & & \XLie \ar[shift left=1]{uu}[near end]{\Ph_{i+1}}\ar[dashed, shift right=3]{lll}[swap]{\XU} \ar[shift right=1]{dl}[rotate=-50]{\vdash}[rotate=-140]{\subset}\\
	\XDi \ar[shift right= 3]{uu}[near end, swap]{\Ps'_{i+1}} \ar[shift left=2]{rrr}{\XLB}[swap]{\top} \ar[dashed, shift left=2.5]{ur}[rotate=40, near end]{\XAS \ } & & & \XLb \ar[shift left=3]{uu}[near end]{\Ps_{i+1}}\ar[shift left=2]{lll}{\XUd} \ar[shift right=3]{ur}[rotate=40, near start, swap]{\Xliel} &
\end{tikzcd}
\]

Besides the previous generalizations, there exist constructions of the actor in the categories of crossed modules of groups, associative algebras and Lie algebras, as well as the construction of an object that under some circumstances is the actor of the category of crossed modules of Leibniz algebras. Therefore, it seems natural to consider the existence of an object in the category of crossed modules of dialgebras which plays the role of the actor, at least under certain conditions.

The article is organized according to the following structure: In Section~\ref{section2} we recall some basic definitions on actions and crossed modules of dialgebras. In Section~\ref{section3} we obtain an object that extends the dialgebra of tetramultipliers~\cite{FCPhD} to the category of crossed modules of dialgebras (Theorem~\ref{theo_action_tetra_LDmu}) and we give a complete description of an action in terms of equations. Then, in Section~\ref{section4} we find sufficient conditions for the object computed in the previous section to be the actor of a crossed module of dialgebras (Theorem~\ref{theo_equiv_XDi_action}). Finally, in Section~\ref{section5} we prove that the center of a crossed module coincides with the kernel of the canonical morphism from a crossed module of dialgebras to its actor.
\section{Basic definitions}\label{section2}
Let us first recall some basic definitions and results as a brief preliminary introduction for the rest of the paper. All algebras are considered over a commutative ring with unit~$\kk$.
\begin{definition}[\cite{Lo_Di}]
	An \emph{associative dialgebra} (or simply \emph{dialgebra}) $D$ is a $\kk$-module together with two bilinear operations { $\dashv, \, \vdash \colon D\times D\rightarrow D$}, the left product and the right product respectively, satisfying the following axioms:
	\begin{align}
		\label{Di1}\tag{Di1} &(x\dashv y)\dashv z = x\dashv (y\vdash z),\\
		\label{Di2}\tag{Di2} &(x\dashv y)\dashv z = x\dashv (y\dashv z),\\
		\label{Di3}\tag{Di3} &(x\vdash y)\dashv z = x\vdash (y\dashv z),\\
		\label{Di4}\tag{Di4} &(x\dashv y)\vdash z = x\vdash (y\vdash z), \\
		\label{Di5}\tag{Di5} &(x\vdash y)\vdash z = x\vdash (y\vdash z),
	\end{align}
	\noindent for all $x,y,z\in D$.
	In some identities we may use $\ast$ to denote both $\vdash$ and $\dashv$, meaning that the corresponding equality is satisfied for $\ast= \; \vdash$ and $\ast= \; \dashv$.
	
	A dialgebra is called \emph{abelian} if both operations are trivial maps.
	A homomorphism of dialgebras is a $\kk$-linear map that preserves both the left and the right products.
\end{definition}

We denote by $\Ann(D)$ the \emph{annihilator} of $D$, that is the subspace of $D$
\begin{equation*}
	\Ann(D)=\{x \in D \ | \ x \ast y = y \ast x = 0, \ \text{for all } y \in D\},
\end{equation*}
which is obviously an ideal of $D$.
\begin{definition}\label{def_action_dial} {(\cite{CFCKhL})}
	Let $D$ and $L$ be dialgebras. \emph{An action} of $D$ on $L$ consists of four linear maps, where two of them are denoted by the symbol $\dashv$ and the other two by $\vdash$,
	\begin{align*}
		\dashv \colon D\otimes L \rightarrow L, \qquad \dashv \colon L\otimes D \rightarrow L, \qquad \vdash \colon D\otimes L \rightarrow L, \qquad \vdash \colon L\otimes D \rightarrow L,
	\end{align*}
	\noindent such that the $30$ equalities resulting from~\eqref{Di1}--\eqref{Di5} by taking one variable in $D$ and the other two in $L$ (15 equalities), and one variable in	$L$ and the other two in $D$ (15 equalities), hold. For example, if $L$ is an ideal of a dialgebra $D$ (maybe $L = D$), then
	the left and right products in $D$ yield an action of $D$ on $L$.
	
	The action is called \emph{trivial} if all these four maps are trivial.
\end{definition}

For a given action of a dialgebra $D$ on a dialgebra $L$, we can consider the \emph{semidirect product} dialgebra
$L \rtimes D$~\cite{CFCKhL}, which consists of the $\kk$-module $L \oplus D$ together with the left and the right products given by
\begin{align*}
	(a_1,x_1)\ast (a_2,x_2)= (a_1 \ast a_2+x_1\ast a_2 + a_1\ast x_2, x_1\ast x_2).
\end{align*}
for all $x_1,x_2 \in D$, $a_1, a_2 \in L$.
\begin{definition} { (\cite{CFCKhL})}
	A \emph{crossed module of dialgebras} $(L,D,\mu)$ is a
	homomorphism of dialgebras~$\mu\colon L \to D$ together with an action of $D$ on ${L}$ such that
	\begin{align*}
		& \mu(x\ast a)=x\ast \mu(a) \quad \text{and} \quad \mu(a \ast x)=\mu(a)\ast x,\label{XDi1}\tag{XDi1}\\
		& \mu(a_1)\ast a_2=a_1\ast a_2=a_1\ast \mu(a_2). \label{XDi2}\tag{XDi2}
	\end{align*}
	for all $x\in D$, $a_1,a_2\in L$.
	
	A \emph{homomorphism of crossed modules of dialgebras} $(\varphi, \psi) \colon (L,D,\mu)\to (L',D', \mu')$
	is a pair of dialgebra homomorphisms, $\varphi \colon {L} \to {L'}$ and $\psi \colon D\to \D'$, such that they commute with $\mu$ and $\mu'$ and they preserve the actions, i.e., $\varphi(x\ast a)=\psi(x)\ast \vp(a)$ and $\vp(a \ast x)=\vp(a) \ast \psi(x)$ for all $x\in D$, $a\in L$.
\end{definition}
Identity~\eqref{XDi1} is usually known as \emph{equivariance}, while~\eqref{XDi2} is called \emph{Peiffer identity}. The category of crossed modules of dialgebras and their homomorphisms will be denoted by $\XDi$.

A common instance of a crossed module of dialgebras is that of a dialgebra $D$ possessing an ideal $L$, where the inclusion homomorphism ${L}\hookrightarrow D$ is a crossed module of dialgebras and the action of $D$ on ${L}$ is given by left and right products in $D$, as in Definition~\ref{def_action_dial}.

Another common example is that of an $D$-bimodule $L$ (see~\cite{Lo_Di} for the definition) considered as an abelian dialgebra, where then the zero homomorphism $0 \colon L\to D$ is a crossed module.

In~\cite{FCPhD}, it is defined the dialgebra of tetramultipliers, as an analogue to the algebra of bimultipliers~\cite{DeLu}, closely related to the notion of multiplication of a ring by Hochschild~\cite{Hoch}, called bimultiplication by Mac Lane~\cite{MacL}. This dialgebra, just like the algebra of bimultipliers, is the actor in {\bf Dias} under certain circumstances {\cite{FCPhD}}. Since our aim is to obtain a $2$-dimensional generalization of the actor in the category of dialgebras, let us first recall the definition of tetramultipliers.
\begin{definition} [\cite{FCPhD}]
	Let $L$ be a dialgebra. A \emph{tetramultiplier} of $L$ is a quadruple $t=(l^{\dashv},r^{\dashv},l^{\vdash},r^{\vdash})$ of $\kk$-linear maps from $L$ to $L$ such that
	\begin{multicols}{3}
		\begin{enumerate}[label=(\arabic*)]
			\item $l^{\dashv}(a \vdash b) = l^{\dashv}(a) \dashv b$,
			\item $l^{\dashv}(a \dashv b) = l^{\dashv}(a) \dashv b$,
			\item $l^{\vdash}(a \dashv b) = l^{\vdash}(a) \dashv b$,
			\item $l^{\vdash}(a \vdash b) = l^{\dashv}(a) \vdash b$,
			\item $l^{\vdash}(a \vdash b) = l^{\vdash}(a) \vdash b$,
			\columnbreak
			\item $r^{\dashv}(a) \dashv b = a \dashv l^{\vdash}(b)$,
			\item $r^{\dashv}(a) \dashv b = a \dashv l^{\dashv}(b)$,
			\item $r^{\vdash}(a) \dashv b = a \vdash l^{\dashv}(b)$,
			\item $r^{\dashv}(a) \vdash b = a \vdash l^{\vdash}(b)$,
			\item $r^{\vdash}(a) \vdash b = a \vdash l^{\vdash}(b)$,
			\columnbreak
			\item $r^{\dashv}(a \dashv b) = a \dashv r^{\vdash}(b)$,
			\item $r^{\dashv}(a \dashv b) = a \dashv r^{\dashv}(b)$,
			\item $r^{\dashv}(a \vdash b) = a \vdash r^{\dashv}(b)$,
			\item $r^{\vdash}(a \dashv b) = a \vdash r^{\vdash}(b)$,
			\item $r^{\vdash}(a \vdash b) = a \vdash r^{\vdash}(b)$,
		\end{enumerate}
	\end{multicols}
	\noindent for all $a, b \in L$.
\end{definition}
We denote by $\Tetra(L)$ the set of all tetramultipliers of $L$. It is a dialgebra with the induced $\kk$-module structure and the left and right products given by
\begin{align*}
	t_1 \dashv t_2
	& = (l_1^{\dashv} l_2^{\vdash},r_2^{\dashv} r_1^{\dashv},l_1^{\vdash} l_2^{\vdash},r_2^{\dashv} r_1^{\vdash}), \\
	t_1 \vdash t_2
	& = (l_1^{\vdash} l_2^{\dashv},r_2^{\dashv} r_1^{\dashv},l_1^{\vdash} l_2^{\vdash},r_2^{\vdash} r_1^{\dashv})
\end{align*}
\noindent for all $t_1=(l_1^{\dashv},r_1^{\dashv},l_1^{\vdash},r_1^{\vdash}), t_2=(l_2^{\dashv},r_2^{\dashv},l_2^{\vdash},r_2^{\vdash}) \in \Tetra(L)$.

It is not difficult to check that, given $a \in L$, the quadruple $(l_a^{\dashv},r_a^{\dashv},l_a^{\vdash},r_a^{\vdash})$, with
\begin{align*}
	l_a^{\dashv}(a') & = a \dashv a', \qquad r_a^{\dashv}(a') = a' \dashv a, \\
	l_a^{\vdash}(a') & = a \vdash a', \qquad r_a^{\vdash}(a') = a' \vdash a,
\end{align*}
\noindent for all $a' \in L$, is a tetramultiplier in $L$.

\begin{remark}
	Let us remark that a dialgebra $L$ with $\vdash ~ = ~ \dashv$ is just an associative algebras. In this case a tetramultiplier is a bimultiplier of $L$ and the dialgebra of tetramultipliers $\Tetra(L)$ becomes the associative algebra of bimultipliers.
\end{remark}

\section{The dialgebra of tetramultipliers}\label{section3}
In this section we extend to { {\sf XDias}} the dialgebra of tetramultipliers. First we need to translate the notion of a tetramultiplier of a dialgebra into a tetramultiplier between two dialgebras via the action.
\begin{definition}\label{def_tetra_DL}
	Given an action of dialgebras of $D$ on $L$, the set of \emph{tetramultipliers} from~$D$ to $L$, denoted by $\Tetra(D,L)$, consists of all the quadruples $(l^{\dashv},r^{\dashv},l^{\vdash},r^{\vdash})$ of $\kk$-linear maps from $D$ to~$L$, such that
	\begin{multicols}{3}
		\begin{enumerate}[label=(\arabic*)]
			\item $l^{\dashv}(x \vdash y) = l^{\dashv}(x) \dashv y$,
			\item $l^{\dashv}(x \dashv y) = l^{\dashv}(x) \dashv y$,
			\item $l^{\vdash}(x \dashv y) = l^{\vdash}(x) \dashv y$,
			\item $l^{\vdash}(x \vdash y) = l^{\dashv}(x) \vdash y$,
			\item $l^{\vdash}(x \vdash y) = l^{\vdash}(x) \vdash y$,
			\columnbreak
			\item $r^{\dashv}(x) \dashv y = x \dashv l^{\vdash}(y)$,
			\item $r^{\dashv}(x) \dashv y = x \dashv l^{\dashv}(y)$,
			\item $r^{\vdash}(x) \dashv y = x \vdash l^{\dashv}(y)$,
			\item $r^{\dashv}(x) \vdash y = x \vdash l^{\vdash}(y)$,
			\item $r^{\vdash}(x) \vdash y = x \vdash l^{\vdash}(y)$,
			\columnbreak
			\item ${r^{\dashv}(x \dashv y) = x \dashv r^{\vdash}(y)}$,
			\item ${r^{\dashv}(x \dashv y) = x \dashv r^{\dashv}(y)}$,
			\item ${r^{\dashv}(x \vdash y) = x \vdash r^{\dashv}(y)}$,
			\item ${r^{\vdash}(x \dashv y) = x \vdash r^{\vdash}(y)}$,
			\item ${r^{\vdash}(x \vdash y) = x \vdash r^{\vdash}(y)}$,
		\end{enumerate}
	\end{multicols}
	\noindent for all $x, y \in D$.
\end{definition}

Given $a \in L$, the quadruple of $\kk$-linear maps $(l_a^{\dashv},r_a^{\dashv},l_a^{\vdash},r_a^{\vdash})$, where
\begin{align*}
	l_a^{\dashv}(x) & = a \dashv x, \qquad r_a^{\dashv}(x) = x \dashv a, \qquad l_a^{\vdash}(x) = a \vdash x, \qquad r_a^{\vdash}(x) = x \vdash a,
\end{align*}
\noindent for all $x \in D$, is clearly a tetramultiplier from $D$ to $L$. Observe that $\Tetra(D,D)$, with the action of $D$ on itself defined by its left and right products, is exactly $\Tetra(D)$.

Let us assume for the rest of the article that $(L,D,\mu)$ is a crossed module of dialgebras.

$\Tetra(D,L)$ has an obvious $\kk$-module structure. Regarding its dialgebra structure, it is described in the next proposition.
\begin{proposition}
	$\Tetra(D,L)$ is a dialgebra with the left and right products given by
	\begin{align}\label{eq_tetra_DL_products}
		t_1 \dashv t_2
		& = (l_1^{\dashv} \mu l_2^{\vdash},r_2^{\dashv} \mu r_1^{\dashv},l_1^{\vdash} \mu l_2^{\vdash},r_2^{\dashv} \mu r_1^{\vdash}), \\
		t_1 \vdash t_2
		& = (l_1^{\vdash} \mu l_2^{\dashv},r_2^{\dashv} \mu r_1^{\dashv},l_1^{\vdash} \mu l_2^{\vdash},r_2^{\vdash} \mu r_1^{\dashv})
	\end{align}
	\noindent for all $t_1=(l_1^{\dashv},r_1^{\dashv},l_1^{\vdash},r_1^{\vdash}), t_2=(l_2^{\dashv},r_2^{\dashv},l_2^{\vdash},r_2^{\vdash}) \in \Tetra(D,L)$.
\end{proposition}
\begin{proof}
	Let $t_1, t_2 \in \Tetra(D,L)$. It follows immediately from~\eqref{XDi1} and~\eqref{XDi2} that $t_1 \ast t_2 \in \Tetra(D,L)$. Checking the corresponding dialgebra identities is just a matter of routine calculations.
\end{proof}
Now we state the following definition.
\begin{definition}\label{def_tetra_LDmu}
	The set of \emph{tetramultipliers of the crossed module of dialgebras} $(L,D,\mu)$, denoted by $\Tetra(L,D,\mu)$, consists of all octuples $((\lam^{\dashv},\rho^{\dashv},\lam^{\vdash},\rho^{\vdash}),(\ka^{\dashv},\om^{\dashv},\ka^{\vdash},\om^{\vdash}))$ such that
	\begin{align}
		& (\lam^{\dashv},\rho^{\dashv},\lam^{\vdash},\rho^{\vdash}) \in \Tetra(L) \quad \text{and} \quad (\ka^{\dashv},\om^{\dashv},\ka^{\vdash},\om^{\vdash}) \in \Tetra(D),\label{axiom_tetra_LDmu_1}\\
		& \mu \lam^{\dashv} = \ka^{\dashv} \mu, \quad \mu \rho^{\dashv} = \om^{\dashv} \mu, \quad \mu \lam^{\vdash} = \ka^{\vdash} \mu, \quad \mu \rho^{\vdash} = \om^{\vdash} \mu\label{axiom_tetra_LDmu_2}
	\end{align}
	\noindent and
	
	\centering{
		\begin{multicols}{3}	
			\item $\lam^{\dashv}(x \vdash a) = \ka^{\dashv}(x) \dashv a$,
			\item $\lam^{\dashv}(x \dashv a) = \ka^{\dashv}(x) \dashv a$,
			\item $\lam^{\vdash}(x \dashv a) = \ka^{\vdash}(x) \dashv a$,
			\item $\lam^{\vdash}(x \vdash a) = \ka^{\dashv}(x) \vdash a$,
			\item $\lam^{\vdash}(x \vdash a) = \ka^{\vdash}(x) \vdash a$,
			\columnbreak
			\item $\rho^{\dashv}(x \dashv a) = x \dashv \rho^{\vdash}(a)$,
			\item $\rho^{\dashv}(x \dashv a) = x \dashv \rho^{\dashv}(a)$,
			\item $\rho^{\dashv}(x \vdash a) = x \vdash \rho^{\dashv}(a)$,
			\item $\rho^{\vdash}(x \dashv a) = x \vdash \rho^{\vdash}(a)$,
			\item $\rho^{\vdash}(x \vdash a) = x \vdash \rho^{\vdash}(a)$,
			\columnbreak
			\item $\rho^{\dashv}(a) \dashv x = a \dashv \ka^{\vdash}(x)$,
			\item $\rho^{\dashv}(a) \dashv x = a \dashv \ka^{\dashv}(x)$,
			\item $\rho^{\vdash}(a) \dashv x = a \vdash \ka^{\dashv}(x)$,
			\item $\rho^{\dashv}(a) \vdash x = a \vdash \ka^{\vdash}(x)$,
			\item $\rho^{\vdash}(a) \vdash x = a \vdash \ka^{\vdash}(x)$,
		\end{multicols}
		
		\begin{multicols}{3}		
			\item $\om^{\dashv}(x) \dashv a = x \dashv \lam^{\vdash}(a)$,
			\item $\om^{\dashv}(x) \dashv a = x \dashv \lam^{\dashv}(a)$,
			\item $\om^{\vdash}(x) \dashv a = x \vdash \lam^{\dashv}(a)$,
			\item $\om^{\dashv}(x) \vdash a = x \vdash \lam^{\vdash}(a)$,
			\item $\om^{\vdash}(x) \vdash a = x \vdash \lam^{\vdash}(a)$,
			\columnbreak
			\item $\lam^{\dashv}(a \vdash x) = \lam^{\dashv}(a) \dashv x$,
			\item $\lam^{\dashv}(a \dashv x) = \lam^{\dashv}(a) \dashv x$,
			\item $\lam^{\vdash}(a \dashv x) = \lam^{\vdash}(a) \dashv x$,
			\item $\lam^{\vdash}(a \vdash x) = \lam^{\dashv}(a) \vdash x$,
			\item $\lam^{\vdash}(a \vdash x) = \lam^{\vdash}(a) \vdash x$,
			\columnbreak
			\item $\rho^{\dashv}(a \dashv x) = a \dashv \om^{\vdash}(x)$,
			\item $\rho^{\dashv}(a \dashv x) = a \dashv \om^{\dashv}(x)$,
			\item $\rho^{\dashv}(a \vdash x) = a \vdash \om^{\dashv}(x)$,
			\item $\rho^{\vdash}(a \dashv x) = a \vdash \om^{\vdash}(x)$,
			\item $\rho^{\vdash}(a \vdash x) = a \vdash \om^{\vdash}(x)$,	
		\end{multicols}
	}
	
	\flushleft{for all $a \in L$, $x \in D$.}
\end{definition}
Given $x \in D$, it can be readily checked that $((\lam_x^{\dashv},\rho_x^{\dashv},\lam_x^{\vdash},\rho_x^{\vdash}),(\ka_x^{\dashv},\om_x^{\dashv},\ka_x^{\vdash},\om_x^{\vdash}))$, where
\begin{align*}
	\lam_x^{\dashv} (a) & = x \dashv a, \quad \rho_x^{\dashv} (a) = a \dashv x, \quad \lam_x^{\vdash} (a) = x \vdash a, \quad \rho_x^{\vdash} (a) = a \vdash x,\\
	\ka_x^{\dashv} (y) & = x \dashv y, \quad \om_x^{\dashv} (y) = y \dashv x, \quad \ka_x^{\vdash} (y) = x \vdash y, \quad \om_x^{\vdash} (y) = y \vdash x,
\end{align*}
is a tetramultiplier of the crossed module $(L,D,\mu)$.

The $\kk$-module structure of $\Tetra(L,D,\mu)$ is evident, while its dialgebra structure is described as follows.
\begin{proposition}
	$\Tetra(L,D,\mu)$ is a dialgebra with the left and right products given by
	\begin{multline}\label{eq_bider_nqmu_bracket}
		((\lam^{\dashv},\rho^{\dashv},\lam^{\vdash},\rho^{\vdash}),(\ka^{\dashv},\om^{\dashv},\ka^{\vdash},\om^{\vdash})) \ast ((\lam^{'\dashv},\rho^{'\dashv},\lam^{'\vdash},\rho^{'\vdash}),(\ka^{'\dashv},\om^{'\dashv},\ka^{'\vdash},\om^{'\vdash})) \\= ((\lam^{\dashv},\rho^{\dashv},\lam^{\vdash},\rho^{\vdash}) \ast (\lam^{'\dashv},\rho^{'\dashv},\lam^{'\vdash},\rho^{'\vdash}),(\ka^{\dashv},\om^{\dashv},\ka^{\vdash},\om^{\vdash}) \ast (\ka^{'\dashv},\om^{'\dashv},\ka^{'\vdash},\om^{'\vdash})),
	\end{multline}
	\noindent for all quadruples $((\lam^{\dashv},\rho^{\dashv},\lam^{\vdash},\rho^{\vdash}),(\ka^{\dashv},\om^{\dashv},\ka^{\vdash},\om^{\vdash})), ((\lam^{'\dashv},\rho^{'\dashv},\lam^{'\vdash},\rho^{'\vdash}),(\ka^{'\dashv},\om^{'\dashv},\ka^{'\vdash},\om^{'\vdash}))$ belonging to $\Tetra(L,D,\mu)$.
\end{proposition}
\begin{proof}
	It follows directly from the dialgebra structures of $\Tetra(L)$ and $\Tetra(D)$.
\end{proof}

Now we aim to define a map $\Delta \colon Tetra(D,L) \to Tetra(L, D, \mu)$ and endow it with a structure of crossed module of dialgebras. This is done in the following theorem.
\begin{theorem}\label{theo_action_tetra_LDmu}
	The $\kk$-linear map $\Delta \colon \Tetra(D,L) \to \Tetra(L,D,\mu)$, given by the assignment $(l^{\dashv},r^{\dashv},l^{\vdash},r^{\vdash}) \mapsto ((l^{\dashv} \mu,r^{\dashv} \mu,l^{\vdash} \mu,r^{\vdash} \mu),(\mu l^{\dashv},\mu r^{\dashv},\mu l^{\vdash},\mu r^{\vdash}))$ is a homomorphism of dialgebras. Additionally, there is an action of $\Tetra(L,D,\mu)$ on $\Tetra(D,L)$ given by:
	\begin{align}
		((\lam^{\dashv},\rho^{\dashv},\lam^{\vdash},\rho^{\vdash}),(\ka^{\dashv},\om^{\dashv},\ka^{\vdash},\om^{\vdash})) \dashv (l^{\dashv},r^{\dashv},l^{\vdash},r^{\vdash}) & = (\lam^{\dashv} l^{\vdash},r^{\dashv} \om^{\dashv},\lam^{\vdash} l^{\vdash},r^{\dashv} \om^{\vdash}),\label{eq_action_Xtetra_left1}\\
		((\lam^{\dashv},\rho^{\dashv},\lam^{\vdash},\rho^{\vdash}),(\ka^{\dashv},\om^{\dashv},\ka^{\vdash},\om^{\vdash})) \vdash (l^{\dashv},r^{\dashv},l^{\vdash},r^{\vdash}) & = (\lam^{\vdash} l^{\dashv},r^{\dashv} \om^{\dashv},\lam^{\vdash} l^{\vdash},r^{\vdash} \om^{\dashv}),\label{eq_action_Xtetra_left2}\\
		(l^{\dashv},r^{\dashv},l^{\vdash},r^{\vdash}) \dashv ((\lam^{\dashv},\rho^{\dashv},\lam^{\vdash},\rho^{\vdash}),(\ka^{\dashv},\om^{\dashv},\ka^{\vdash},\om^{\vdash}))& = (l^{\dashv} \ka^{\vdash},\rho^{\dashv} r^{\dashv},l^{\vdash} \ka^{\vdash},\rho^{\dashv} r^{\vdash}),\label{eq_action_Xtetra_right1}\\
		(l^{\dashv},r^{\dashv},l^{\vdash},r^{\vdash}) \vdash ((\lam^{\dashv},\rho^{\dashv},\lam^{\vdash},\rho^{\vdash}),(\ka^{\dashv},\om^{\dashv},\ka^{\vdash},\om^{\vdash}))& = (l^{\vdash} \ka^{\dashv},\rho^{\dashv} r^{\dashv},l^{\vdash} \ka^{\vdash},\rho^{\vdash} r^{\dashv}),\label{eq_action_Xtetra_right2}
	\end{align}
	\noindent for all $((\lam^{\dashv},\rho^{\dashv},\lam^{\vdash},\rho^{\vdash}),(\ka^{\dashv},\om^{\dashv},\ka^{\vdash},\om^{\vdash})) \in \Tetra(L,D,\mu)$, $(l^{\dashv},r^{\dashv},l^{\vdash},r^{\vdash}) \in \Tetra(D,L)$.
	The homomorphism of dialgebras $\Delta$ together with the action above is a crossed module of dialgebras.
\end{theorem}
\begin{proof}
	Given $(l^{\dashv},r^{\dashv},l^{\vdash},r^{\vdash}) \in \Tetra(D,L)$, it is easy to check that $(l^{\dashv} \mu,r^{\dashv} \mu,l^{\vdash} \mu,r^{\vdash} \mu) \in \Tetra(L)$ and $(\mu l^{\dashv},\mu r^{\dashv},\mu l^{\vdash},\mu r^{\vdash}) \in \Tetra(D)$. Therefore $\Delta$ is well defined. It is a matter of straightforward calculations to prove that it is a homomorphism of dialgebras.
	
	Checking that the action is well defined follows directly from combining the properties from Definition~\ref{def_tetra_DL} and Definition~\ref{def_tetra_LDmu}. As for the 30 equalities from Definition~\ref{def_action_dial}, they can be readily checked by using the properties satisfied by the left and the right products in $\Tetra(D,L)$ and $\Tetra(L,D,\mu)$.
\end{proof}

\begin{remark}\label{remark_Delta}
	Let $D$ and $L$ be associative algebras with the multiplications both denoted by $\cdot$. Let us consider them as dialgebras with the equal left and right multiplications, i.e., $\vdash ~ = ~\cdot ~ = ~ \dashv $. Then the action of $D$ on $L$ is an action of associative algebras and a crossed module of dialgebras $(L,D,\mu)$ is a crossed module of associative algebras. Furthermore, $\Tetra(L, D, \mu)$ coincides with the associative algebra of bimultipliers (see \cite[Definition~2.4]{BoCaDaUs}), whilst $\Tetra(D, L)$ coincides with the associative algebra of crossed bimultipliers of $(L, D, \mu)$ (see \cite[Definition~2.6]{BoCaDaUs}). In addition, if $D$ and $L$ satisfy certain extra conditions, the map (crossed module of associative algebras) $\Delta \colon \Tetra(D,L) \to \Tetra(L,D,\mu)$ is the actor (or split extension classifier) of the crossed module $(L, D, \mu)$ of associative algebras (see \cite[Theorem 5.6]{BoCaDaUs}).
\end{remark}

\section{The actor}\label{section4}

Categories of interest, fistly described by Orzech in~\cite{Or}, are categories of groups with operations together with two more conditions which emulate some sort of associativity. It can be proved that $\Di$ is indeed a category of interest, whilst $\XDi$ is not. However, in~\cite{AsCeUs} it is checked that $\XDi$ is equivalent to $cat^1$-dialgebras, which satisfy the conditions of modified categories of interest (see~\cite{BoCaDaUs}). Therefore, it is reasonable to consider representability of actions in $\XDi$ similarly to what is done for crossed modules of associative algebras in~\cite{BoCaDaUs}.
Nevertheless, $\XDi$ is also a semi-abelian {category}, for which actions are equivalent to split extensions \cite[Lemma 1.3]{BoJaKe}. For this reason, we prefer to address the problem in a different way, from a more combinatorial perspective, via the semidirect product (split extension) of crossed modules of dialgebras.

We consider the \emph{actor} (as used in~\cite{BoCaDaUs, CaDaLa}) or \emph{split extension classifier} (as described in~\cite{BoJaKe}) as an object that represents actions in a given semi-abelian category.

It is important to note that, for a dialgebra $D$, $\Tetra(D)$ is the actor of $D$, but only if some conditions are satisfied, as it is proved in the next result from~\cite{FCPhD}.
\begin{proposition}\label{prop_Di_ann_perfect}
	Let $D$ be a dialgebra such that $\Ann(D)=0$ or $D \dashv D = D =D \vdash D$. Then $\Tetra(D)$ is the actor of $D$.
\end{proposition}

Under the assumption of some conditions, it is reasonable to take into consideration the crossed module $(\Tetra(D,L),\Tetra(L,D,\mu),\Delta)$ as a suitable option for actor in $\XDi$ (see Proposition~\ref{prop_Di_ann_perfect}). Nevertheless, given two crossed modules of dialgebras, $(M,P,\eta)$ and $(L,D,\mu)$, it would be too audacious to define an action of $(M,P,\eta)$ on $(L,D,\mu)$ directly as a homomorphism from $(M,P,\eta)$ to $(\Tetra(D,L),\Tetra(L,D,\mu),\Delta)$, as we have no option to guarantee that this homomorphism results into a set of actions of $(M,P,\eta)$ on $(L,D,\mu)$.

Actions of crossed modules of different structures have also been described in terms of equations, as it can be checked in~\cite{CaInKhLa} for crossed modules of groups, in~\cite{CaCaKhLa},~\cite{ChShZh} or~\cite{TaSh} for crossed modules of Lie algebras, and in~\cite{CaFCGMKh} for crossed modules of Leibniz algebras. Bearing those examples in mind, it is seemingly reasonable to address the problem by following the next steps: firstly, we consider a homomorphism from $(M,P,\eta)$ to $\ol{\Act}(L,D,\mu) =(\Tetra(D,L),\Tetra(L,D,\mu),\Delta)$, and translate into equations every property satisfied by that homomorphism. As expected, then we need to find the conditions for the existence of that collection of equations to be equivalent to the existence of a homomorphism from $(M,P,\eta)$ to $\ol{\Act}(L,D,\mu)$. The final step is to construct the semidirect product in order to check that these equations induce a set of actions.	

\begin{theorem}\label{theo_equiv_XDi_action}
	Let $(M,P,\eta)$ and $(L,D,\mu)$ in $\XDi$. It can be constructed a homomorphism of crossed modules from $(M,P,\eta)$ to $(\Tetra(D,L), \Tetra(L,D,\mu), \Delta)$, if the following conditions are satisfied:
	\begin{itemize}
		\item [\rm (i)] There exist actions of the dialgebra $P$ (and therefore $M$) on the dialgebras $L$ and $D$. The homomorphism $\mu$
		is $P$-equivariant, that is
		\begin{align}
			\mu(p \ast a) & =p \ast \mu(a), \label{p_equivariant_Di_1} \tag{DiEQ1} \\
			\mu(a \ast p) & =\mu(a) \ast p, \label{p_equivariant_Di_2} \tag{DiEQ2}
		\end{align}
		\noindent and the actions of $P$ and $D$ on $L$ are compatible, that is the 30 equalities resulting from~\eqref{Di1}--\eqref{Di5} by considering all the possible combinations taking one element in $P$, one in~$D$ and one in $L$.
		
		\item[\rm (ii)] There are four $\kk$-bilinear maps $\xi^{\dashv}_1 \colon M\times D\to L$, $\xi^{\vdash}_1 \colon M\times D\to L$, $\xi^{\dashv}_2 \colon D\times M\to L$ and $\xi^{\vdash}_2 \colon D\times M\to L$ such that
		\begin{align}
			\mu \xi^{\dashv}_2(x,m) & =x \dashv m ,\label{action_Di_1a} \tag{DiM1a} \\
			\mu \xi^{\vdash}_2(x,m) & =x \vdash m ,\label{action_Di_1b} \tag{DiM1b} \\
			\mu \xi^{\dashv}_1(m,x) & = m \dashv x,\label{action_Di_1c} \tag{DiM1c} \\
			\mu \xi^{\vdash}_1(m,x) & = m \vdash x,\label{action_Di_1d} \tag{DiM1d}
		\end{align}
		\begin{align}		
			\xi^{\dashv}_2 (\mu(a),m) & = a \dashv m,\label{action_Di_2a} \tag{DiM2a} \\
			\xi^{\vdash}_2 (\mu(a),m) & = a \vdash m,\label{action_Di_2b} \tag{DiM2b} \\
			\xi^{\dashv}_1 (m,\mu(a)) & = m \dashv a,\label{action_Di_2c} \tag{DiM2c} \\
			\xi^{\vdash}_1 (m,\mu(a)) & = m \vdash a,\label{action_Di_2d} \tag{DiM2d}
		\end{align}
		\begin{align}	
			\xi^{\dashv}_2 (x,p \dashv m) & = \xi^{\dashv}_2 (x \dashv p,m) = \xi^{\dashv}_2 (x,p \vdash m), \label{action_Di_3a} \tag{DiM3a} \\
			\xi^{\vdash}_2 (x,p \dashv m) & = \xi^{\dashv}_2 (x \vdash p,m), \label{action_Di_3b} \tag{DiM3b} \\
			\xi^{\vdash}_2 (x,p \vdash m) & = \xi^{\vdash}_2 (x \dashv p,m), \label{action_Di_3c} \tag{DiM3c} \\
			\xi^{\dashv}_1 (p \dashv m, x) & = p \dashv \xi^{\vdash}_1(m,x), \label{action_Di_3d} \tag{DiM3d} \\
			\xi^{\dashv}_1 (p \vdash m, x) & = p \vdash \xi^{\dashv}_1(m,x), \label{action_Di_3e} \tag{DiM3e} \\
			\xi^{\vdash}_1 (p \dashv m, x) & = p \vdash \xi^{\vdash}_1(m,x) = \xi^{\vdash}_1 (p \vdash m, x), \label{action_Di_3f} \tag{DiM3f} \\
			\xi^{\dashv}_2 (x,m \dashv p) & = \xi^{\dashv}_2 (x,m) \dashv p = \xi^{\dashv}_2 (x,m \vdash p), \label{action_Di_3g} \tag{DiM3g} \\
			\xi^{\vdash}_2 (x,m \dashv p) & = \xi^{\vdash}_2 (x,m) \dashv p, \label{action_Di_3h} \tag{DiM3h} \\
			\xi^{\vdash}_2 (x,m \vdash p) & = \xi^{\dashv}_2 (x,m) \vdash p, \label{action_Di_3i} \tag{DiM3i} \\
			\xi^{\dashv}_1 (m \dashv p, x) & = \xi^{\dashv}_1 (m, p \vdash x), \label{action_Di_3j} \tag{DiM3j} \\
			\xi^{\dashv}_1 (m \vdash p, x) & = \xi^{\vdash}_1 (m, p \dashv x), \label{action_Di_3k} \tag{DiM3k} \\
			\xi^{\vdash}_1 (m \dashv p, x) & = \xi^{\vdash}_1 (m, p \vdash x) = \xi^{\vdash}_1 (m \vdash p, x), \label{action_Di_3l} \tag{DiM3l}
		\end{align}
		\begin{align}	
			\xi^{\dashv}_1 (m \dashv m', x) & = m \dashv \xi^{\vdash}_1(m', x), \label{action_Di_4a} \tag{DiM4a} \\
			\xi^{\dashv}_1 (m \vdash m', x) & = m \vdash \xi^{\dashv}_1(m', x), \label{action_Di_4b} \tag{DiM4b} \\
			\xi^{\vdash}_1 (m \dashv m', x) & = m \vdash \xi^{\vdash}_1(m', x) = \xi^{\vdash}_1 (m \vdash m', x), \label{action_Di_4c} \tag{DiM4c} \\
			\xi^{\dashv}_2 (x, m \dashv m') & = \xi^{\dashv}_2 (x, m) \dashv m' = \xi^{\dashv}_2 (x, m \vdash m'), \label{action_Di_4d} \tag{DiM4d} \\
			\xi^{\vdash}_2 (x, m \dashv m') & = \xi^{\vdash}_2 (x, m) \dashv m', \label{action_Di_4e} \tag{DiM4e} \\
			\xi^{\vdash}_2 (x, m \vdash m') & = \xi^{\dashv}_2 (x, m) \vdash m', \label{action_Di_4f} \tag{DiM4f}
		\end{align}
		\begin{align}			
			\xi^{\dashv}_1 (m, x \vdash y) & = \xi^{\dashv}_1 (m, x) \dashv y = \xi^{\dashv}_1 (m, x \dashv y) , \label{action_Di_5a} \tag{DiM5a} \\
			\xi^{\vdash}_1 (m, x \dashv y) & = \xi^{\vdash}_1 (m, x) \dashv y, \label{action_Di_5b} \tag{DiM5b} \\
			\xi^{\vdash}_1 (m, x \vdash y) & = \xi^{\dashv}_1 (m, x) \vdash y = \xi^{\vdash}_1 (m, x) \vdash y , \label{action_Di_5c} \tag{DiM5c} \\
			\xi^{\dashv}_2 (x, m) \dashv y & = x \dashv \xi^{\vdash}_1 (m, y) = x \dashv \xi^{\dashv}_1 (m, y) , \label{action_Di_5d} \tag{DiM5d} \\
			\xi^{\vdash}_2 (x, m) \dashv y & = x \vdash \xi^{\dashv}_1 (m, y), \label{action_Di_5e} \tag{DiM5e} \\
			\xi^{\dashv}_2 (x, m) \vdash y & = x \vdash \xi^{\vdash}_1 (m, y) = \xi^{\vdash}_2 (x, m) \vdash y , \label{action_Di_5f} \tag{DiM5f} \\
			\xi^{\dashv}_2 (x \dashv y, m) & = x \dashv \xi^{\vdash}_2 (y, m) = x \dashv \xi^{\dashv}_2 (y, m) , \label{action_Di_5g} \tag{DiM5g} \\
			\xi^{\dashv}_2 (x \vdash y, m) & = x \vdash \xi^{\dashv}_2 (y, m), \label{action_Di_5h} \tag{DiM5h} \\
			\xi^{\vdash}_2 (x \dashv y, m) & = x \vdash \xi^{\vdash}_2 (y, m) = \xi^{\vdash}_2 (x \vdash y, m) , \label{action_Di_5i} \tag{DiM5i}
		\end{align}
		\noindent for all $m,m'\in M$, $a\in L$, $p\in P$, $x,x'\in D$.
	\end{itemize}
	
	In addition, if one of the following conditions hold the converse statement is also true:
	\begin{align}
		& \Ann(L)=0=\Ann(D), \label{condition1} \tag{CON1} \\
		& \Ann(L)=0 \quad \text{and} \quad D \dashv D = D = D \vdash D, \label{condition2} \tag{CON2} \\
		& L \dashv L = L = L \vdash L \quad \text{and} \quad D \dashv D = D = D \vdash D, \label{condition3} \tag{CON3} \\
		& L \dashv L = L = L \vdash L \quad \text{and} \quad \Ann(D)=0. \label{condition4} \tag{CON4}
	\end{align}
\end{theorem}
\begin{proof}
	Assume that (i) and (ii) hold. Then, we can define a homomorphism of crossed modules $(\vp,\psi)$ from $(M,P,\eta)$ to $\ol{\Act}(L,D,\mu)$ as follows. For any $m \in M$, $\vp(m)=(l^{\dashv}_{m},r^{\dashv}_{m},l^{\vdash}_{m},r^{\vdash}_{m})$, with
	\begin{equation*}
		l^{\dashv}_{m} (x) = \xi^{\dashv}_{1}(m,x), \qquad r^{\dashv}_{m} (x) = \xi^{\dashv}_{2}(x,m), \qquad l^{\vdash}_{m} (x) = \xi^{\vdash}_{1}(m,x), \qquad r^{\vdash}_{m} (x) = \xi^{\vdash}_{2}(x,m),
	\end{equation*}
	\noindent for all $x \in D$. On the other hand, for any $p \in P$, $\psi(p)=((\lam_{p}^{\dashv},\rho_{p}^{\dashv},\lam_{p}^{\vdash},\rho_{p}^{\vdash}),(\ka_{p}^{\dashv},\om_{p}^{\dashv},\ka_{p}^{\vdash},\om_{p}^{\vdash}))$, with
	\begin{align*}
		\lam_{p}^{\dashv}(a) & = p \dashv a, \qquad \rho_{p}^{\dashv}(a) = a \dashv p, \qquad \lam_{p}^{\vdash}(a) = p \vdash a, \qquad \rho_{p}^{\vdash}(a) = a \vdash p,\\
		\ka_{p}^{\dashv}(x) & = p \dashv x, \qquad \om_{p}^{\dashv}(x) = x \dashv p, \qquad \ka_{p}^{\vdash}(x) = p \vdash x, \qquad \om_{p}^{\vdash}(x) = x \vdash p,
	\end{align*}
	\noindent for all $a \in L$, $x \in D$.
	It follows directly from~\eqref{action_Di_5a}--\eqref{action_Di_5i} that $(l^{\dashv}_{m},r^{\dashv}_{m},l^{\vdash}_{m},r^{\vdash}_{m}) \in \Tetra(D,L)$ for all $m \in M$. Besides, $\vp$ is $\kk$-linear and for any $m,m' \in M$,
	\begin{align*}
		\vp(m) \dashv \vp(m') & = (l^{\dashv}_{m} \mu l^{\vdash}_{m'},r^{\dashv}_{m'} \mu r^{\dashv}_{m},l^{\vdash}_{m} \mu l^{\vdash}_{m'},r^{\dashv}_{m'} \mu r^{\vdash}_{m}), \\
		\vp(m) \vdash \vp(m') & = (l^{\vdash}_{m} \mu l^{\dashv}_{m'},r^{\dashv}_{m'} \mu r^{\dashv}_{m},l^{\vdash}_{m} \mu l^{\vdash}_{m'},r^{\vdash}_{m'} \mu r^{\dashv}_{m}),
	\end{align*}
	\noindent while
	\begin{align*}
		\vp(m \dashv m') & = (l^{\dashv}_{m \dashv m'} , r^{\dashv}_{m \dashv m'} , l^{\vdash}_{m \dashv m'} , r^{\vdash}_{m \dashv m'}), \\
		\vp(m \vdash m') & = (l^{\dashv}_{m \vdash m'} , r^{\dashv}_{m \vdash m'} , l^{\vdash}_{m \vdash m'} , r^{\vdash}_{m \vdash m'}).
	\end{align*}
	In order to show that $\varphi$ preserves both $\vdash$ and $\dashv$, the respective eight equalities can be readily checked by making use of the identities~\eqref{action_Di_2a}--\eqref{action_Di_2d} and~\eqref{action_Di_4a}--\eqref{action_Di_4f}.
	
	Regarding $\psi$, it is necessary to show that $\psi(p)$ satisfies the axioms from Definition~\ref{def_tetra_LDmu} for any $p \in P$. The fact that $(\lam_{p}^{\dashv},\rho_{p}^{\dashv},\lam_{p}^{\vdash},\rho_{p}^{\vdash})$ (respectively $(\ka_{p}^{\dashv},\om_{p}^{\dashv},\ka_{p}^{\vdash},\om_{p}^{\vdash})$) is a tetramultiplier of~$L$ (respectively $D$) follows immediately from the actions of $P$ on $L$ and $D$. The identities $\mu \lam^{\dashv} = \ka^{\dashv} \mu$, $\mu \rho^{\dashv} = \om^{\dashv} \mu$, $\mu \lam^{\vdash} = \ka^{\vdash} \mu$ and $\mu \rho^{\vdash} = \om^{\vdash} \mu$ are consequences of~\eqref{p_equivariant_Di_1} and~\eqref{p_equivariant_Di_2} respectively; while the remaining 30 identities in Definition~\ref{def_tetra_LDmu} follow from the compatibility of the actions of $P$ and $D$ on $L$. Moreover, due to the actions of $P$ on $D$ and $L$, $\psi$ is a homomorphism of dialgebras.
	
	Recall that
	\begin{align*}
		\Delta \vp (m) & = ((l^{\dashv}_m \mu, r^{\dashv}_m \mu, l^{\vdash}_m \mu, r^{\vdash}_m \mu),(\mu l^{\dashv}_m ,\mu r^{\dashv}_m ,\mu l^{\vdash}_m ,\mu r^{\vdash}_m)), \\
		\psi\eta (m) & = ((\lam^{\dashv}_{\eta(m)}, \rho^{\dashv}_{\eta(m)}, \lam^{\vdash}_{\eta(m)}, \rho^{\vdash}_{\eta(m)}),(\ka^{\dashv}_{\eta(m)} ,\om^{\dashv}_{\eta(m)} ,\ka^{\vdash}_{\eta(m)} ,\om^{\vdash}_{\eta(m)})),
	\end{align*}
	\noindent for any $m \in M$. The equality $\Delta \vp = \psi \eta$ can be readily checked from~\eqref{action_Di_1a}--\eqref{action_Di_1d},~\eqref{action_Di_2a}--\eqref{action_Di_2d} and the action of $M$ on $L$ and $D$ via $P$.
	
	Finally, we only need to check the behaviour of $(\vp,\psi)$ with the action of $P$ on $M$. Let $m \in M$ and $p \in P$. By~\eqref{eq_action_Xtetra_left1}--\eqref{eq_action_Xtetra_right2},
	\begin{align*}
		\psi(p) \dashv \vp(m) & = (\lam^{\dashv}_p l^{\vdash}_m,r^{\dashv}_m \om^{\dashv}_p,\lam^{\vdash}_p l^{\vdash}_m,r^{\dashv}_m \om^{\vdash}_p),\\
		\psi(p) \vdash \vp(m) & = (\lam^{\vdash}_p l^{\dashv}_m,r^{\dashv}_m \om^{\dashv}_p,\lam^{\vdash}_p l^{\vdash}_m,r^{\vdash}_m \om^{\dashv}_p),\\
		\vp(m) \dashv \psi(p) & = (l^{\dashv}_m \ka^{\vdash}_p,\rho^{\dashv}_p r^{\dashv}_m,l^{\vdash}_m \ka^{\vdash}_p,\rho^{\dashv}_p r^{\vdash}_m),\\
		\vp(m) \vdash \psi(p) & = (l^{\vdash}_m \ka^{\dashv}_p,\rho^{\dashv}_p r^{\dashv}_m,l^{\vdash}_m \ka^{\vdash}_p,\rho^{\vdash}_p r^{\dashv}_m),
	\end{align*}
	On the other hand, we know by definition that
	\begin{align*}
		\vp(p \dashv m) & = (l^{\dashv}_{p \dashv m},r^{\dashv}_{p \dashv m}, l^{\vdash}_{p \dashv m}, r^{\vdash}_{p \dashv m}), \\
		\vp(p \vdash m) & = (l^{\dashv}_{p \vdash m},r^{\dashv}_{p \vdash m}, l^{\vdash}_{p \vdash m}, r^{\vdash}_{p \vdash m}), \\
		\vp(m \dashv p) & = (l^{\dashv}_{m \dashv p},r^{\dashv}_{m \dashv p}, l^{\vdash}_{m \dashv p}, r^{\vdash}_{m \dashv p}), \\
		\vp(m \vdash p) & = (l^{\dashv}_{m \vdash p},r^{\dashv}_{m \vdash p}, l^{\vdash}_{m \vdash p}, r^{\vdash}_{m \vdash p}),
	\end{align*}
	\noindent Directly from~\eqref{action_Di_3a}--\eqref{action_Di_3l} one can easily see that the required identities between components hold. Then, $(\vp,\psi)$ is a homomorphism of crossed modules of dialgebras.
	
	Now let us see that at least one of the conditions~\eqref{condition1}--\eqref{condition4} is necessary to prove the converse statement. Assume that there is a homomorphism of crossed modules
	\begin{equation}\label{action_diagramDi}
		\xymatrix {
			M \ar[d]_{\vp}\ar[r]^{\eta} & P \ar[d]^{\psi} \\
			\Tetra(D,L) \ar[r]_{\Delta} & \Tetra(L, D, \mu)
		}
	\end{equation}
	For any $m \in M$ and $p \in P$, we denote $\psi(p)$ by $((\lam^{\dashv}_{p},\rho^{\dashv}_{p},\lam^{\vdash}_{p},\rho^{\vdash}_{p}),(\ka^{\dashv}_{p},\om^{\dashv}_{p},\ka^{\vdash}_{p},\om^{\vdash}_{p}))$ and $\vp(m)$ by $(l^{\dashv}_{m},r^{\dashv}_{m},l^{\vdash}_{m},r^{\vdash}_{m})$, satisfying the conditions from Definition~\ref{def_tetra_LDmu} and from Definition~\ref{def_tetra_DL} respectively. Moreover, by the definition of $\Delta$ (Proposition~\ref{theo_action_tetra_LDmu}), the commutativity of~\eqref{action_diagramDi} can be seen by the identity
	\begin{multline}\label{commutativity_action_Di}
		((\lam^{\dashv}_{\eta(m)},\rho^{\dashv}_{\eta(m)},\lam^{\vdash}_{\eta(m)},\rho^{\vdash}_{\eta(m)}),(\ka^{\dashv}_{\eta(m)},\om^{\dashv}_{\eta(m)},\ka^{\vdash}_{\eta(m)},\om^{\vdash}_{\eta(m)})) \\ = ((l^{\dashv}_{m} \mu,r^{\dashv}_{m} \mu,l^{\vdash}_{m} \mu,r^{\vdash}_{m} \mu),(\mu l^{\dashv}_{m},\mu r^{\dashv}_{m},\mu l^{\vdash}_{m},\mu r^{\vdash}_{m})),
	\end{multline}
	\noindent for all $m \in M$.
	
	We can define eight bilinear maps via $\psi$, two from $P \times L$ on $L$, two from $L \times P$ to $L$, two from $P \times D$ to $D$ and two from $D \times P$ to $D$, as follows:
	\begin{align*}
		p \dashv a & = \lam_{p}^{\dashv}(a), \qquad a \dashv p = \rho_{p}^{\dashv}(a), \qquad p \dashv x = \ka_{p}^{\dashv}(x), \qquad x \dashv p = \om_{p}^{\dashv}(x),\\
		p \vdash a & = \lam_{p}^{\vdash}(a), \qquad a \vdash p = \rho_{p}^{\vdash}(a), \qquad p \vdash x = \ka_{p}^{\vdash}(x), \qquad x \vdash p = \om_{p}^{\vdash}(x),
	\end{align*}
	for all $a\in L$, $x \in D$, $p \in P$. Since $(\lam^{\dashv}_{p},\rho^{\dashv}_{p},\lam^{\vdash}_{p},\rho^{\vdash}_{p}) \in \Tetra(L)$, $(\ka^{\dashv}_{p},\om^{\dashv}_{p},\ka^{\vdash}_{p},\om^{\vdash}_{p}) \in \Tetra(D)$ and $\psi$ is a homomorphism of dialgebras, it can be easily checked that some equalities stated in Definition~\ref{def_action_dial} are satisfied by the previous bilinear maps, but others are not. In particular, for all~$a \in L$, $p,q \in P$, we get the following:
	\begin{align}
		(a \vdash p) \vdash q = \rho^{\vdash}_{q} \rho^{\vdash}_{p} (a) & \neq \rho^{\vdash}_{q} \rho^{\dashv}_{p} (a) = \rho^{\vdash}_{p \vdash q} = a \vdash (p \vdash q), \label{eq_Di_action_20}\tag{D2.2a}	\\
		(p \dashv a) \dashv q = \rho^{\dashv}_{q} \lam^{\dashv}_{p} (a) & \neq \lam^{\dashv}_{p} \rho^{\vdash}_{q} (a) = p \dashv (a \vdash q), \label{eq_Di_action_21}\tag{D2.2b}	\\
		(p \dashv a) \dashv q = \rho^{\dashv}_{q} \lam^{\dashv}_{p} (a) & \neq \lam^{\dashv}_{p} \rho^{\dashv}_{q} (a) = p \dashv (a \dashv q), \label{eq_Di_action_22}\tag{D2.2c}	\\
		(p \vdash a) \dashv q = \rho^{\dashv}_{q} \lam^{\vdash}_{p} (a) & \neq \lam^{\vdash}_{p} \rho^{\dashv}_{q} (a) = p \vdash (a \dashv q), \label{eq_Di_action_23}\tag{D2.2d}	\\
		(p \dashv a) \vdash q = \rho^{\vdash}_{q} \lam^{\dashv}_{p} (a) & \neq \lam^{\vdash}_{p} \rho^{\vdash}_{q} (a) = p \vdash (a \vdash q), \label{eq_Di_action_24}\tag{D2.2d}	\\
		(p \vdash a) \vdash q = \rho^{\vdash}_{q} \lam^{\vdash}_{p} (a) & \neq \lam^{\vdash}_{p} \rho^{\vdash}_{q} (a) = p \vdash (a \vdash q), \label{eq_Di_action_25}\tag{D2.2e}	\\
		(p \dashv q) \dashv a = \lam^{\dashv}_{p \dashv q} = \lam^{\dashv}_{p} \lam^{\vdash}_{q} (a) & \neq \lam^{\dashv}_{q} \lam^{\dashv}_{p} (a) = p \dashv (q \dashv q), \label{eq_Di_action_27}\tag{D2.2f}		
	\end{align}
	\noindent and the analogous inequalities for all $x \in D$, $p, q \in P$, with $\ka$ and $\om$ instead of $\lam$ and $\rho$ respectively. Nevertheless, if at least one of the conditions~\ref{condition1}--\ref{condition4} holds, then either $\Ann(L)= 0$ or $L \dashv L = L = L \vdash L$ and, simultaneously, either $\Ann(D)= 0$ or~$D \dashv D = D = D \vdash D$. Conditions on $L$ (respectively~$D$) guarantee the equalities~\eqref{eq_Di_action_20}--\eqref{eq_Di_action_25} and~\eqref{eq_Di_action_27} for $\lam$ and $\rho$ (respectively $\ka$ and $\om$), so that $P$ acts on $L$ (respectively~$D$). Calculations are fairly straightforward and follow from the properties satisfied by tetramultipliers of $L$ (respectively $D$). Let us assume that $\Ann(L)=0$. Therefore it would be enough to prove that $\rho^{\vdash}_{q} \rho^{\vdash}_{p} (a) - \rho^{\vdash}_{q} \rho^{\dashv}_{p} (a)$, $\rho^{\dashv}_{q} \lam^{\dashv}_{p} (a) - \lam^{\dashv}_{p} \rho^{\vdash}_{q} (a)$, $\rho^{\dashv}_{q} \lam^{\dashv}_{p} (a) - \lam^{\dashv}_{p} \rho^{\dashv}_{q} (a)$, $\rho^{\dashv}_{q} \lam^{\vdash}_{p} (a) - \lam^{\vdash}_{p} \rho^{\dashv}_{q} (a)$, $\rho^{\vdash}_{q} \lam^{\dashv}_{p} (a) - \lam^{\vdash}_{p} \rho^{\vdash}_{q} (a)$, $\rho^{\vdash}_{q} \lam^{\vdash}_{p} (a) - \lam^{\vdash}_{p} \rho^{\vdash}_{q} (a)$, $\lam^{\dashv}_{p} \lam^{\vdash}_{q} (a) - \lam^{\dashv}_{q} \lam^{\dashv}_{p} (a) \in \Ann(L)$. As an example, we show here that $\rho^{\dashv}_{q} \lam^{\vdash}_{p} (a) - \lam^{\vdash}_{p} \rho^{\dashv}_{q} (a) \in \Ann(L)$. Let $b \in L$. Then
	\begin{equation*}
		\begin{split}
			& \rho^{\dashv}_{q} \lam^{\vdash}_{p} (a) \dashv b - \lam^{\vdash}_{p} \rho^{\dashv}_{q}(a) \dashv b \\
			\overset{(7)}{=} & \lam^{\vdash}_{p}(a) \dashv \lam^{\dashv}_{q}(b) - \lam^{\vdash}_{p} \rho^{\dashv}_{q}(a) \dashv b \\
			\overset{(3)}{=} & \lam^{\vdash}_{p}(a) \dashv \lam^{\dashv}_{q}(b) - \lam^{\vdash}_{p} (\rho^{\dashv}_{q}(a) \dashv b) \\
			\overset{(7)}{=} & \lam^{\vdash}_{p}(a) \dashv \lam^{\dashv}_{q}(b) - \lam^{\vdash}_{p} (a \dashv \lam^{\dashv}_{q}(b)) \\
			\overset{(3)}{=} & \lam^{\vdash}_{p}(a) \dashv \lam^{\dashv}_{q}(b) - \lam^{\vdash}_{p} (a) \dashv \lam^{\dashv}_{q}(b) = 0,
		\end{split}
		\qquad \qquad
		\begin{split}
			& \rho^{\dashv}_{q} \lam^{\vdash}_{p} (a) \vdash b - \lam^{\vdash}_{p} \rho^{\dashv}_{q}(a) \vdash b \\
			\overset{(9)}{=} & \lam^{\vdash}_{p} (a) \vdash \lam^{\vdash}_{q}(b) - \lam^{\vdash}_{p} \rho^{\dashv}_{q}(a) \vdash b \\
			\overset{(5)}{=} & \lam^{\vdash}_{p} (a) \vdash \lam^{\vdash}_{q}(b) - \lam^{\vdash}_{p} (\rho^{\dashv}_{q}(a) \vdash b) \\
			\overset{(9)}{=} & \lam^{\vdash}_{p} (a) \vdash \lam^{\vdash}_{q}(b) - \lam^{\vdash}_{p} (a \vdash \lam^{\vdash}_{q}(b)) \\
			\overset{(5)}{=} & \lam^{\vdash}_{p} (a) \vdash \lam^{\vdash}_{q}(b) - \lam^{\vdash}_{p} (a) \vdash \lam^{\vdash}_{q}(b) = 0,
		\end{split}
	\end{equation*}
	\\
	\begin{equation*}
		\begin{split}
			& b \dashv \rho^{\dashv}_{q} \lam^{\vdash}_{p} (a) - b \dashv \lam^{\vdash}_{p} \rho^{\dashv}_{q}(a) \\
			\overset{(6)}{=} & b \dashv \rho^{\dashv}_{q} \lam^{\vdash}_{p} (a) - \rho^{\dashv}_{p}(b) \dashv \rho^{\dashv}_{q}(a) \\
			\overset{(12)}{=} & \rho^{\dashv}_{q}(b \dashv \lam^{\vdash}_{p} (a)) - \rho^{\dashv}_{p}(b) \dashv \rho^{\dashv}_{q}(a) \\
			\overset{(6)}{=} & \rho^{\dashv}_{q}(\rho^{\dashv}_{p}(b) \dashv a) - \rho^{\dashv}_{p}(b) \dashv \rho^{\dashv}_{q}(a) \\
			\overset{(12)}{=} & \rho^{\dashv}_{p}(b) \dashv \rho^{\dashv}_{q}(a) - \rho^{\dashv}_{p}(b) \dashv \rho^{\dashv}_{q}(a) = 0,
		\end{split}
		\qquad \qquad
		\begin{split}
			& b \vdash \rho^{\dashv}_{q} \lam^{\vdash}_{p} (a) - b \vdash \lam^{\vdash}_{p} \rho^{\dashv}_{q}(a) \\
			\overset{(9)}{=} & b \vdash \rho^{\dashv}_{q} \lam^{\vdash}_{p} (a) - \rho^{\dashv}_{p}(b) \vdash \rho^{\dashv}_{q}(a) \\
			\overset{(13)}{=} & \rho^{\dashv}_{q}(b \vdash \lam^{\vdash}_{p} (a)) - \rho^{\dashv}_{p}(b) \vdash \rho^{\dashv}_{q}(a) \\
			\overset{(9)}{=} & \rho^{\dashv}_{q}(\rho^{\dashv}_{p}(b) \vdash a) - \rho^{\dashv}_{p}(b) \vdash \rho^{\dashv}_{q}(a) \\
			\overset{(13)}{=} & \rho^{\dashv}_{p}(b) \vdash \rho^{\dashv}_{q}(a) - \rho^{\dashv}_{p}(b) \vdash \rho^{\dashv}_{q}(a) = 0.
		\end{split}
	\end{equation*}
	
	On the other hand, if we assume that $L \dashv L = L = L \vdash L$, any element $a$ in $L$ can be expressed either as a linear combination of left products $b \dashv c$ or right products $b \vdash c$ in~$L$. Bearing that in mind, it would be sufficient to show that the identities
	\begin{align*}
		& \rho^{\vdash}_{q} \rho^{\vdash}_{p} (a) = \rho^{\vdash}_{q} \rho^{\dashv}_{p} (a), \ \rho^{\dashv}_{q} \lam^{\dashv}_{p} (a) = \lam^{\dashv}_{p} \rho^{\vdash}_{q} (a), \
		\rho^{\dashv}_{q} \lam^{\dashv}_{p} (a) = \lam^{\dashv}_{p} \rho^{\dashv}_{q} (a), \rho^{\dashv}_{q} \lam^{\vdash}_{p} (a) = \lam^{\vdash}_{p} \rho^{\dashv}_{q} (a), \\
		& \rho^{\vdash}_{q} \lam^{\dashv}_{p} (a) = \lam^{\vdash}_{p} \rho^{\vdash}_{q} (a), \ \rho^{\vdash}_{q} \lam^{\vdash}_{p} (a) = \lam^{\vdash}_{p} \rho^{\vdash}_{q} (a) \ \ \text{and} \ \
		\lam^{\dashv}_{p} \lam^{\vdash}_{q} (a) = \lam^{\dashv}_{q} \lam^{\dashv}_{p} (a)
	\end{align*}
	\noindent hold when $a$ is either of the form $b \dashv c$ or $b \vdash c$, given $p, q \in P$. Actually, straightforward calculations, using the conditions satisfied by $\psi(p)$ and $\psi(q)$, show that:
	\begin{align*}
		& \rho^{\vdash}_{q} \rho^{\vdash}_{p}(b \dashv c) = \rho^{\vdash}_{q} \rho^{\dashv}_{p}(b \dashv c), \ \rho^{\dashv}_{q} \lam^{\dashv}_{p}(b \ast c) = \lam^{\dashv}_{p} \rho^{\vdash}_{q}(b \ast c), \
		\rho^{\dashv}_{q} \lam^{\dashv}_{p}(b \ast c) = \lam^{\dashv}_{p} \rho^{\dashv}_{q} (b \ast c), \\
		& \rho^{\dashv}_{q} \lam^{\vdash}_{p}(b \ast c) = \lam^{\vdash}_{p} \rho^{\dashv}_{q}(b \ast c), \ \rho^{\vdash}_{q} \lam^{\dashv}_{p}(b \ast c) = \lam^{\vdash}_{p} \rho^{\vdash}_{q}(b \ast c), \ \rho^{\vdash}_{q} \lam^{\vdash}_{p}(b \ast c) = \lam^{\vdash}_{p} \rho^{\vdash}_{q}(b \ast c) \\
		& \text{and} \ \ 	\lam^{\dashv}_{p} \lam^{\vdash}_{q}(b \vdash c) = \lam^{\dashv}_{q} \lam^{\dashv}_{p} (b \vdash c),
	\end{align*}
	\noindent but in general $\rho^{\vdash}_{q} \rho^{\vdash}_{p}(b \vdash c) \neq \rho^{\vdash}_{q} \rho^{\dashv}_{p}(b \vdash c)$ and $\lam^{\dashv}_{p} \lam^{\vdash}_{q}(b \dashv c) = \lam^{\dashv}_{q} \lam^{\dashv}_{p} (b \dashv c)$, for all $b, c \in L$. Therefore the result would not be necessarily true under the hypothesis $L \dashv L = L$ or~$L \vdash L=L$.
	
	Concerning~\eqref{p_equivariant_Di_1} and~\eqref{p_equivariant_Di_2}, they are direct consequences from~\eqref{axiom_tetra_LDmu_2} (recall that, by hypothesis, $((\lam^{\dashv}_{p},\rho^{\dashv}_{p},\lam^{\vdash}_{p},\rho^{\vdash}_{p}),(\ka^{\dashv}_{p},\om^{\dashv}_{p},\ka^{\vdash}_{p},\om^{\vdash}_{p}))$ is a tetramultiplier of $(L,D,\mu)$ for any~${p \in P}$). Similarly, compatibility of the actions of $P$ and $D$ on $L$ follows almost immediately from the 30 identities stated in Definition~\ref{def_tetra_LDmu}. Hence, (i) holds.
	
	Regarding (ii), we can define $\xi^{\dashv}_1(m,x)= l^{\dashv}_{m}(x)$, $\xi^{\vdash}_1(m,x)= l^{\vdash}_{m}(x)$, $\xi^{\dashv}_2(x,m)= r^{\dashv}_{m}(x)$ and~$\xi^{\vdash}_2(x,m)= r^{\vdash}_{m}(x)$ for all $m \in M$, $x \in D$. Then, $\xi_1$ and $\xi_2$ are clearly bilinear. Identities~\eqref{action_Di_1a}--\eqref{action_Di_1d} and~\eqref{action_Di_2a}--\eqref{action_Di_2d} follow automatically from the identity~\eqref{commutativity_action_Di} together with the fact that the actions of $M$ on $L$ and $D$ are induced by the actions of $P$ via $\eta$.
	
	Identities~\eqref{action_Di_5a}--\eqref{action_Di_5h} are just a direct consequence of the fact that, by hypothesis, $(l^{\dashv}_{m},r^{\dashv}_{m},l^{\vdash}_{m},r^{\vdash}_{m})$ is a tetramultiplier from $D$ to $L$ for any $m \in M$.
	
	Since $\vp$ is a homomorphism of dialgebras, we know that
	\begin{align*}
		(l^{\dashv}_{m \dashv n},r^{\dashv}_{m \dashv n},l^{\vdash}_{m \dashv n},r^{\vdash}_{m \dashv n}) & = (l_m^{\dashv} \mu l_n^{\vdash},r_n^{\dashv} \mu r_m^{\dashv},l_m^{\vdash} \mu l_n^{\vdash},r_n^{\dashv} \mu r_m^{\vdash}), \\
		(l^{\dashv}_{m \dashv n},r^{\dashv}_{m \dashv n},l^{\vdash}_{m \dashv n},r^{\vdash}_{m \dashv n}) & = (l_m^{\vdash} \mu l_n^{\dashv},r_n^{\dashv} \mu r_m^{\dashv},l_m^{\vdash} \mu l_n^{\vdash},r_n^{\vdash} \mu r_m^{\dashv}).
	\end{align*}
	{These identities}, together with~\eqref{action_Di_2a}--\eqref{action_Di_2d}, allow us to easily prove that~\eqref{action_Di_4a}--\eqref{action_Di_4f} hold.
	
	Recall that since $(\vp,\psi)$ is a homomorphism of crossed modules { of} dialgebras, we have:
	\begin{align*}
		(l^{\dashv}_{p \dashv m},r^{\dashv}_{p \dashv m}, l^{\vdash}_{p \dashv m}, r^{\vdash}_{p \dashv m}) & = (\lam^{\dashv}_p l^{\vdash}_m,r^{\dashv}_m \om^{\dashv}_p,\lam^{\vdash}_p l^{\vdash}_m,r^{\dashv}_m \om^{\vdash}_p),\\
		(l^{\dashv}_{p \vdash m},r^{\dashv}_{p \vdash m}, l^{\vdash}_{p \vdash m}, r^{\vdash}_{p \vdash m}) & = (\lam^{\vdash}_p l^{\dashv}_m,r^{\dashv}_m \om^{\dashv}_p,\lam^{\vdash}_p l^{\vdash}_m,r^{\vdash}_m \om^{\dashv}_p),\\
		(l^{\dashv}_{m \dashv p},r^{\dashv}_{m \dashv p}, l^{\vdash}_{m \dashv p}, r^{\vdash}_{m \dashv p}) & = (l^{\dashv}_m \ka^{\vdash}_p,\rho^{\dashv}_p r^{\dashv}_m,l^{\vdash}_m \ka^{\vdash}_p,\rho^{\dashv}_p r^{\vdash}_m),\\
		(l^{\dashv}_{m \vdash p},r^{\dashv}_{m \vdash p}, l^{\vdash}_{m \vdash p}, r^{\vdash}_{m \vdash p}) & = (l^{\vdash}_m \ka^{\dashv}_p,\rho^{\dashv}_p r^{\dashv}_m,l^{\vdash}_m \ka^{\vdash}_p,\rho^{\vdash}_p r^{\dashv}_m),
	\end{align*}
	Identities~\eqref{action_Di_3a}--\eqref{action_Di_3l} follow immediately from the previous identities. Hence, (ii) holds.
\end{proof}
\begin{example}\label{canonical_morph}
	Let $(M,P,\eta) \in \XDi$, there is a homomorphism $(\vp,\psi) \colon (M,P,\eta) \to \ol{\Act}(M,P,\eta)$, with $\vp(m) = (l^{\dashv}_{m},r^{\dashv}_{m},l^{\vdash}_{m},r^{\vdash}_{m})$ and $\psi(p) = ((\lam^{\dashv}_{p},\rho^{\dashv}_{p},\lam^{\vdash}_{p},\rho^{\vdash}_{p}),(\ka^{\dashv}_{p},\om^{\dashv}_{p},\ka^{\vdash}_{p},\om^{\vdash}_{p}))$, where
	\begin{equation*}
		l^{\dashv}_{m} (p) = p \dashv m, \qquad r^{\dashv}_{m} (p) = m \dashv p, \qquad l^{\vdash}_{m} (p) = p \vdash m, \qquad r^{\vdash}_{m} (p) = m \vdash p
	\end{equation*}
	\noindent and
	\begin{align*}
		\lam^{\dashv}_{p} (m) & = m \dashv p, \qquad \rho^{\dashv}_{p} (m) = p \dashv m, \qquad \lam^{\vdash}_{p} (m) = m \vdash p, \qquad \rho^{\vdash}_{p} (m) = p \vdash m, \\
		\ka^{\dashv}_{p} (q) & = q \dashv p, \qquad \om^{\dashv}_{p} (q) = p \dashv q, \qquad \ka^{\vdash}_{p} (q) = q \vdash p, \qquad \om^{\vdash}_{p} (q) = p \vdash q,
	\end{align*}
	\noindent for all $m \in M$, $p,q \in P$. The computations in order to prove that $(\vp,\psi)$ is a homomorphism of crossed modules of dialgebras are fairly straightforward. This homomorphism does not necessarily define a set of actions valid to construct the semidirect product. Nevertheless, Theorem~\ref{theo_equiv_XDi_action}, together with the result immediately bellow, show that if $(M,P,\eta)$ satisfies at least one of the conditions~\eqref{condition1}--\eqref{condition4}, then the previous homomorphism does define an appropriate set of actions of $(M,P,\eta)$ on itself.
\end{example}
Let us consider $(M,P,\eta)$ and $(L,D,\mu)$ to be crossed modules of dialgebras such that~(i) and (ii) from Theorem~\ref{theo_equiv_XDi_action} hold. Then, there are actions of $M$ on $L$ and of $P$ on $D$, so we can consider the semidirect products of dialgebras $L \rtimes M$ and $D \rtimes P$. Moreover, we get the following theorem.
\begin{theorem}\label{theo_semi_XDi}
	There is an action of the dialgebra $D \rtimes P$ on the dialgebra $L \rtimes M$, defined by
	\begin{align}
		(x,p) \dashv (a,m) & = (x \dashv a + p \dashv a + \xi^{\dashv}_2(x,m),p \dashv m), \label{action_semidirect_XDi_left_1} \\
		(x,p) \vdash (a,m) & = (x \vdash a + p \vdash a + \xi^{\vdash}_2(x,m),p \vdash m), \label{action_semidirect_XDi_left_2} \\
		(a,m) \dashv (x,p) & = (a \dashv x + a \dashv p + \xi^{\dashv}_1(m,x),m \dashv p), \label{action_semidirect_XDi_right_1} \\
		(a,m) \vdash (x,p) & = (a \vdash x + a \vdash p + \xi^{\vdash}_1(m,x),m \vdash p), \label{action_semidirect_XDi_right_2}
	\end{align}
	\noindent for all $(x,p) \in D \rtimes P$, $(a,m) \in L \rtimes M$, with $\xi^{\dashv}_1$, $\xi^{\vdash}_1,$ $\xi^{\dashv}_2$ and $\xi^{\vdash}_2$ as in Theorem~\ref{theo_equiv_XDi_action}. Furthermore, the homomorphism of dialgebras $(\mu,\eta) \colon L \rtimes M \to D \rtimes P$, defined by
	\[
	(\mu,\eta)(a,m)=(\mu(a),\eta(m)),
	\]
	\noindent for any $(a,m)\in L \rtimes M$, together with the above action, is a crossed module of dialgebras.
\end{theorem}
\begin{proof}
	The 30 identities that have to be checked in order to prove that~\eqref{action_semidirect_XDi_left_1}--\eqref{action_semidirect_XDi_right_2} define an action of dialgebras follow simply from the conditions satisfied by the crossed modules~$(M,P,\eta)$ and~$(L,D,\mu)$ (see Theorem~\ref{theo_equiv_XDi_action}). In any case, as an example, we will show how to prove the third one. The computations for the rest of the identities are similar. Let~$(a,m),(b,n') \in L \rtimes M$ and~$(x,p) \in D \rtimes P$. Then, we get that
	\begin{multline*}
		((x,p) \dashv (a,m)) \dashv (b,n) = \big(\underbrace{(x \dashv a) \dashv b}_{(1)} + \underbrace{(p \dashv a) \dashv b}_{(2)} + \underbrace{\xi^{\dashv}_2(x,m) \dashv b}_{(3)} + \underbrace{(p \dashv m) \dashv b}_{(4)} \\ + \underbrace{(x \dashv a) \dashv n}_{(5)} + \underbrace{(p \dashv a) \dashv n}_{(6)} + \underbrace{\xi^{\dashv}_2(x,m) \dashv n}_{(7)}, \underbrace{(p \dashv m) \dashv n}_{(8)}\big),\\
		(x,p) \dashv ((a,m) \vdash (b,n)) = \big(\underbrace{x \dashv (a \vdash b)}_{(1')} + \underbrace{p \dashv (a \vdash b)}_{(2')} + \underbrace{x \dashv (m \vdash b)}_{(3')} + \underbrace{p \dashv (m \vdash b)}_{(4')} \\ + \underbrace{x \dashv (a \vdash n)}_{(5')} + \underbrace{p \dashv (a \vdash n)}_{(6')} + \underbrace{\xi^{\dashv}_2(x,m \vdash n) \dashv n}_{(7')}, \underbrace{p \dashv (m \vdash n)}_{(8')}\big).
	\end{multline*}
	
	We want to show that $(i) = (i')$, for $i=1,\dots,8$. It is immediate for $i=1,2,8$ due to the action of~$D$ on~$L$ and the actions of~$P$ on~$L$ and~$M$. For $i=7$, it suffices to use~\eqref{action_Di_4d}. For $i=3$, the identity is deduced from the fact that the action of $M$ on $L$ and $D$ is defined via $\eta$ together with the compatibility of the action of~$P$ and~$D$ on~$L$,~\eqref{action_Di_1a} and the Peiffer identity of~$\mu$. Namely,
	\begin{align*}
		x \dashv (m \vdash b) & = x \dashv (\eta(m) \vdash b) = (x \dashv \eta(m)) \dashv b = (x \dashv m) \dashv b\\
		& = \mu \xi^{\dashv}_2(x,m) \dashv b = \xi^{\dashv}_2(x,m) \dashv b.
	\end{align*}
	For $i=4$, it is necessary to use of the equivariance of $\eta$, the definition of the action of $M$ on $L$ via $\eta$ and the action of $P$ on $L$:
	\begin{align*}
		p \dashv (m \vdash b) & = p \dashv (\eta(m) \vdash b) = (p \dashv \eta(m)) \dashv b = \eta(p \dashv m) \dashv b = (p \dashv m) \dashv b
	\end{align*}
	For $i=5$ and $i=6$ the equalities can be easily proved by using the definition of the action of $M$ on $L$ via $\eta$ and the compatibility of the action of $P$ and $D$ on $L$.
	
	Seeing that $(\mu,\eta)$ is a homomorphism of dialgebras follows immediately from the definition of the action of $M$ on $L$ via $\eta$ together with conditions~\eqref{p_equivariant_Di_1} and~\eqref{p_equivariant_Di_2}. Concerning the equivariance of $(\mu,\eta)$, for any $(a,m)\in L \rtimes M$ and $(x,p) \in D \rtimes P$,
	\begin{align*}
		(\mu,\eta) ((x,p) \dashv (a,m)) & = (\mu,\eta) (x \dashv a + p \dashv a + \xi^{\dashv}_2(x,m), p \dashv m)\\
		& = (\mu (x \dashv a) + \mu(p \dashv a) + \mu\xi^{\dashv}_2(x,m), \eta(p \dashv m))\\
		& = (x \dashv \mu(a) + p \dashv \mu(a) + x \dashv \eta(m), p \dashv \eta(m))\\
		& = (x, p) \dashv (\mu (a), \eta(m))\\
		& = (x, p) \dashv ((\mu, \eta) (a,m)),
	\end{align*}
	\noindent by the equivariance of $\mu$ and $\eta$,~\eqref{p_equivariant_Di_1},~\eqref{action_Di_1a} and the definition of the action of $M$ on $D$ via $\eta$. The remaining identities can be proved similarly.
	
	The Peiffer identity of the pair $(\mu,\eta)$ follows directly from the homonymous property of~$\mu$ and $\eta$, the definition of the action of $M$ on $L$ via~$\eta$ and conditions~\eqref{action_Di_2a}--\eqref{action_Di_2d}.
\end{proof}
\begin{definition}
	The crossed module of dialgebras $(L \rtimes M, D \rtimes P,(\mu, \eta))$ is called the
	\emph{semidirect product} of the crossed modules of dialgebras $(L,D,\mu)$ and $(M,P,\eta)$.
\end{definition}
Note that, as usual, from the semidirect product we can determine a split extension of $(M,P,\eta)$ by $(L,D,\mu)$
\[
\xymatrix {
	(0,0,0) \ar[r] & (L,D,\mu) \ar[r] & (L \rtimes M, D \rtimes P,(\mu, \eta)) \ar@<0.7ex>[r] & (M,P,\eta) \ar[r] \ar@<0.7ex>[l] & (0,0,0)
}
\]
Conversely, any split extension of $(M,P,\eta)$ by $(L,D,\mu)$ is isomorphic to
its semidirect product, where the action of $(M,P,\eta)$ on $(L,D,\mu)$ is induced by the section.

\begin{remark}
	Let $(M,P,\eta)$ and $(L,D,\mu)$ be crossed modules of dialgebras satisfying at least one of the conditions (CON1)-(CON 4). Then,
	an action of $(M,P,\eta)$ on $(L,D,\mu)$ can also be defined as a homomorphism of crossed modules of dialgebras from $(M,P,\eta)$ to $\ol{\Act}(L,D,\mu)$. This means that, under one of those conditions, $\ol{\Act}(L,D,\mu)$ is the actor of $(L,D,\mu)$ and it can be denoted simply by $\Act(L,D,\mu)$.
	
	It is also relevant to note that the generalization of the actor to the category of crossed modules of dialgebras is somehow smoother than the one in the case of crossed modules of Leibniz algebras, as in this case there is no need of the extra conditions in Theorem~\ref{theo_equiv_XDi_action}~(ii) (see~\cite{CaFCGMKh, FCPhD}) for the equivalent set of conditions in the case of Leibniz algebras and a fourth possible condition makes the converse statement true, which does not work in the case of Leibniz algebras, that is $L \dashv L = L = L \vdash L$ and $\Ann(D)=0$.
\end{remark}

\section{Center of a crossed module of dialgebras}\label{section5}

As a conclusion to all the previous considerations and results, let us consider $(L,D,\mu)$ as a crossed module of dialgebras satisfying at least one of the conditions~\eqref{condition1}--\eqref{condition4}. Let $\Z(D)$ be the center of the dialgebra $D$, which in this case coincides with its annihilator (although the center and the annihilator may not be equal in general). We consider the canonical homomorphism $(\vp,\psi)$ from $(L,D,\mu)$ to $\Act(L,D,\mu)$, as in Example~\ref{canonical_morph}. It follows that
\begin{align*}
	\Ker(\vp) & = L^{D} \quad \text{and} \quad	\Ker(\psi) = \st_{D}(L) \cap \Z(D),
\end{align*}
where
\begin{align*}
	L^{D} &= \{a \in L \, | \, x \dashv a = x \vdash a = a \dashv x = a \vdash x = 0, \ \text{for all} \ x \in D\} \\
	\st_{D}(L) &= \{x \in D \, | \, x \dashv a = x \vdash a = a \dashv x = a \vdash x = 0, \ \text{for all} \ a \in L\}
\end{align*}

Thus, the kernel of $(\vp,\psi)$ is the crossed module of dialgebras $(L^{D}, \st_{D}(L) \cap \Z(D),\mu)$. Therefore, the kernel of $(\vp,\psi)$ is exactly the center of the crossed module $(L,D,\mu)$, as defined in \cite[Definition 27]{AsCeUs} for crossed modules in modified categories of interest. This is in fact the categorical notion of center defined by Huq~\cite{Huq} confirming that our construction of the actor for a crossed module of dialgebras is consistent.

\begin{example} Let $(L,D,\iota)$ be the crossed module where $L$ is an ideal of $D$ and $\iota$ is the inclusion. Then, its center is obtained by the crossed module $(L \cap \Z(D), \Z(D), \iota)$. In particular, the center of $(0,D,0)$ is $(0, \Z(D), 0)$ and the center of $(D,D,\id_{D})$ is $(\Z(D), \Z(D), \id)$.
\end{example}


\end{document}